\documentclass{article}
\usepackage{amsthm}
\usepackage{amsfonts}
\usepackage{amsmath}
\usepackage{amscd}
\usepackage{amssymb}
\begin{document}
\newcommand{\np}{\addtocounter{problem}{-1}}
\newcommand{\C}{\mathbb{C}}
\newcommand{\R}{\mathbb{R}}
\newcommand{\Z}{\mathbb{Z}}
\newcommand{\N}{\mathbb{N}}
\newcommand{\aut}{\text{Aut}}
\title{Feasibly Reducing KAT Equations to KA Equations}
\author{James Worthington\\Mathematics Department, Cornell University \\ Ithaca, NY 14853-4201 USA \\ 
{\tt worthing@math.cornell.edu} }  
\maketitle
\theoremstyle{definition}
\newtheorem{problem}{Problem}
\newtheorem{rudin}{Rudin}

\begin{abstract}  
Kleene algebra (KA) is the algebra of regular events.  Familiar examples of Kleene algebras include regular sets, relational algebras, and trace algebras.  A Kleene algebra with tests (KAT) is a Kleene algebra with an embedded Boolean subalgebra.  The addition of tests allows one to encode {\tt while} programs as KAT terms, thus the equational theory of KAT can express (propositional) program equivalence.  More complicated statements about programs can be expressed in the Hoare theory of KAT, which suffices to encode Propositional Hoare Logic.

That the equational theory of KAT reduces to the equational theory of KA has been shown by Cohen et al.  Unfortunately, their reduction
involves an exponential blowup in the size of the terms involved.  Here we give an alternate feasible reduction.
\end{abstract}

\section {Introduction}

The class of Kleene algebras is defined by equations and equational implications over the signature $\{0, 1, +, \cdot,  ^*\}$.  Some well-known examples of Kleene algebras include relational algebras, trace algebras, and sets of regular languages (see \cite{bib:be03} for more examples and applications).  In fact, the set of regular languages over an alphabet $\Sigma$ is the free Kleene algebra on $\Sigma$.  That is, given two KA terms $\alpha$ and $\beta$, $\alpha = \beta$ modulo the axioms of Kleene algebra if and only if $\alpha$ and $\beta$ denote the same regular set \cite{bib:ko94}.  A Kleene algebra with tests is a Kleene algebra with an embedded Boolean subalgebra (the complementation function is only defined on Boolean terms).  

Adding tests allows the encoding of {\tt while} programs as KAT terms.  As a result, the equational theory of KAT suffices to express (propositional) equivalence of {\tt while} programs.  Moreover, Propositional Hoare Logic can be encoded in the Hoare theory of KAT (equational implications of the form $r = 0 \rightarrow p = q$), and furthermore the Hoare theory of KAT reduces efficiently to the equational theory of KAT.  Combining all of these reductions shows that the equational theory of KA can be used to express interesting properties of programs succinctly.  See \cite{bib:ko97}, \cite{bib:ko00}, and \cite{bib:ko01} for details.

In \cite{bib:ko96}, it is shown that the equational theory of KAT reduces to the equational theory of KA.  Unfortunately, the reduction used can increase the size of the terms involved exponentially.  We given alternate reduction, which increases the size of the terms by only a polynomial amount.  
This paper is organized as follows.  In section 2, we provide the relevant definitions and recall the encoding of finite automata as Kleene algebra terms.  In section 3, we prove some useful theorems of Kleene algebra used for reasoning about automata and give an overview
of {\it guarded string algebras}. In section 4, we give a feasible reduction from a KAT term to an automaton encoded as a KA term.
In section 5, we remark that the Hoare theory of KA(T) can be efficiently reduced to the equational theory of KA(T), and in section 6 we
make an observation concerning automata constructed from KAT terms representing deterministic {\tt while} programs.

\section{Background}
In this section, we describe our proof system and recall some useful facts about KA(T).  The axiomatization of Kleene algebra, results about matrices, and the encoding of automata as KA terms are from \cite{bib:ko94}.  The definition of KAT is from \cite{bib:ko97}.
 
\subsection{Equational Logic}

By ``proof'', we mean a sequent in the equational implication calculus. Let $\alpha, \beta, \gamma, \delta$ be terms in the
language of Kleene algebra. The equational axioms are:
$$
\begin{array}{l}
\alpha = \alpha\\
\alpha = \beta \rightarrow \beta = \alpha\\
\alpha = \beta \rightarrow \beta = \gamma \rightarrow \alpha = \gamma\\
\alpha = \beta \rightarrow \gamma = \delta \rightarrow \alpha + \gamma = \beta + \delta\\
\alpha = \beta \rightarrow \gamma = \delta \rightarrow \alpha \cdot \gamma = \beta \cdot \delta\\
\alpha = \beta \rightarrow \alpha^* = \beta^*.\\
\end{array}
$$

\noindent We consider these Horn formulas to be implicitly universally quantified.\\

\noindent Let $\Phi$ be a sequence of equations or equational implications, $e$ an equation, $\phi$ a Horn formula, and $\psi$ an equational axiom or an axiom of KA (given below).  Let $\sigma$ be a substitution of terms for variables.  The rules of inference are:

$$\vdash \sigma(\psi)~~~~~~~
e \vdash e ~~~~~~~~
\begin{array}{c}
\Phi \vdash \phi \\ \hline
\Phi,e \vdash \phi \\ \end{array}~~~~~~~~
\begin{array}{c}
\Phi,e \vdash \phi  \\ \hline
\Phi \vdash e \rightarrow \phi \\ \end{array}~~~~~~~~
\begin{array}{c}
\Phi \vdash e ~~~~~~~\Phi \vdash e \rightarrow \phi \\ \hline
\Phi \vdash \phi \\ \end{array},
$$

\noindent and the structural rules which allow us to treat a sequence of formulas as a set of formulas.  For a proof that this is a complete deductive system, see \cite{bib:se72}.  We also allow ``substitution of equals for equals''.  For example, from $a=b$, conclude
$c(a+1)=c(b+1)$ in one step.   

\subsection {Kleene Algebra}

We now state the axioms of Kleene algebra.  The first are the idempotent semiring axioms.  Note that we abbreviate $\alpha \cdot \beta$ as
$\alpha \beta$.

\begin{enumerate}
\item $(a + b) + c = a + (b + c)$
\item $a + b = b + a$
\item $a + 0 = a$
\item $a + a = a$
\item $(ab)c = a(bc)$
\item $1a = a1 = a$
\item $a(b + c) = ab + ac$
\item $(a + b)c = ac + bc$
\item $0a = 0a = 0$
\end{enumerate}

\noindent In any idempotent semiring, addition can be used to define a partial order: 
$$x \leq y \Leftrightarrow x + y = y.$$

\noindent For brevity, we add the symbol $\leq$ to the language.\\  

\noindent There are four axioms involving $^*$.  The equational axioms are:\\

10. $1 + xx^* = x^*$

11. $1 + x^*x = x^*$\\

\noindent There are also two equational implications:\\

12. $b + ax \leq  x \rightarrow a^* b \leq x$

13. $b + xa \leq x \rightarrow ba^* \leq x$\\

\noindent The equational implications guarantee unique least solutions to the linear inequalities
$$b + aX \leq X$$
$$b + Xa \leq X$$

\noindent in the presence of the other axioms.

\subsection {Kleene Algebra with Tests}

A Kleene algebra with tests is a Kleene algebra with an embedded Boolean subalgebra; Boolean terms are called tests.  Formally, a Kleene algebra with tests is a two-sorted structure $(K,B,+,\cdot,^*,^-,0,1)$ such that $(K,+,\cdot,^*,0,1)$ is a Kleene algebra and 
$(B,+,\cdot,^-,0,1)$ is a Boolean algebra.  Note that complementation is only defined on tests. 

We use the following axiomatization of Boolean algebra.  Let $b,c,d$ be Boolean terms.
\begin{enumerate} 
\item KA axioms 1 - 9 
\item $\overline{0} = 1;~ \overline{1} = 0$
\item $b + 1 = 1$
\item $b\overline{b} = \overline{b}b= 0$
\item $\overline{\overline{b}} = b$
\item $bb=b$
\item $\overline{b + c} = \overline{b}\overline{c};~ \overline{bc} = \overline{b} + \overline{c}$
\item $bc= cb$
\item $b + cd = (b+c)(b+d)$
\end{enumerate}

Any Boolean term $b$ satisfies $b \leq 1$.  Since $1^* = 1$ and $^*$ is monotonic, the KA axioms imply $b^* = 1$.  Note that any Kleene algebra can be viewed as a KAT with $\{0,1\}$ as the two-element Boolean subalgebra.

\subsection {Matrices and Automata}
\noindent The Kleene algebra axioms imply that the set of $n \times n$ matrices over a KA also forms a KA.  Addition and multiplication of matrices are defined in the usual way, 0 is interpreted as the $n \times n$ zero matrix, and 1 as $I_n$.  Equality and the partial order $\leq$ are defined componentwise. To define the star of an $n \times n$ matrix, we first define the star of a $2 \times 2$ matrix:

$$ \left[ 
\begin{array}{lr}
a & b \\
c & d  \\ \end{array}
\right]^*  =
\left[
\begin{array}{lr}
(a + bd^*c)^* & (a + bd^*c)bd^*\\
(d + ca^*b)ca^* & (d + ca^*b)^*\\ \end{array}
\right].
$$
 
\noindent We then extend this definition to arbitrary square matrices inductively.  Given a square matrix $E$, partition $E$ into four submatrices
$$ E = \left[ 
\begin{array}{l|r}
A & B \\ \hline
C & D  \\ \end{array}
\right]  
$$

\noindent such that $A$ and $D$ are square.  By induction, $A^*$ and $D^*$ exist.  Let $F = A + BD^*C$.  Then
$$ E^* = \left[ 
\begin{array}{c|c}
F^* & F^*BD^* \\ \hline
D^*CF^* & D^* + D^*CF^*BD^*  \\ \end{array}
\right].  
$$
\noindent It is a consequence of the KA axioms that any partition may be chosen to compute $E^*$.\\

In \cite{bib:co96}, it is shown that the set of $n \times n$ matrices over a Kleene algebra with tests is a Kleene algebra with tests.  The Boolean subalgebra is the set of matrices with Boolean terms on the diagonal and all other entries equal to 0.

At several points in the proof below, we will have to reason about non-square matrices.  We would like to know whether the 
theorems of Kleene algebra hold when the primitive letters are interpreted as matrices of arbitrary dimension and the function
symbols are treated polymorphically.  In general, the answer is no.  However, there is a large class of theorems for which this does hold, and
they suffice for our purposes.  See \cite{bib:ko98} for a thorough treatment of this issue.

We now recall how to use matrices over a KA to encode finite automata.
 
\newtheorem{definition}{Definition}
\begin{definition}
{\it An} automaton {\it over a Kleene algebra K is a triple $(u,A,v)$ where $u$ and $v$ are $n$-dimensional vectors with entries from 
$\{0,1\}$ and $A$ is an $n \times n$ matrix over $K$.  The vector $u$ encodes the start states of the automaton and is called the} start 
vector. {\it The vector $v$ encodes the accept states of the automaton and is called the} accept vector. {\it The matrix $A$ is called the}
transition matrix.  {\it The language accepted
by $(u,A,v)$ is $u^{\rm{T}}A^*v$.  The} size {\it of $(u,A,v)$ is the number of states, i.e., if $A$ is an $n \times n$ matrix,
then the size of $(u,A,v)$ is $n$}.
\end{definition}

This definition is a bit general for the purposes at hand.  Given an alphabet $\Sigma$, let $\mathcal{F}_{\Sigma}$ be the free Kleene algebra
on generators $\Sigma$.  Over $\mathcal{F}_{\Sigma}$, the definition of an automaton given above is essentially the same as the classical
definition of a finite automaton.  In the sequel, all automata are over some $\mathcal{F}_{\Sigma}$.  Furthermore, most of the automata we
consider have uncomplicated transition matrices.

\begin{definition}
{\it Let $(u,A,v)$ be an automaton over $\mathcal{F}_{\Sigma}$.  The automaton $(u,A,v)$ is} simple {\it if $A$ can be expressed as a sum

$$A = J + \sum_{a \in \Sigma} a \cdot A_a$$

\noindent where $J$ and each $A_a$ is a 0-1 matrix.

\noindent The automaton $(u,A,v)$ is} $\epsilon$-free {\it if $J$ is the zero matrix.

\noindent The automaton $(u,A,v)$ is} deterministic {\it if it is simple, $\epsilon$-free, and $u$ and all rows of each $A_a$ have exactly one 1.}

\end{definition}

Given an automaton $(u,A,v)$, we denote the transition relation encoded by $A$ as $\delta_A$, and the extended transition
relation defined on (states,words) as $\hat{\delta}_A$.  Given an $a \in \Sigma$, we denote the restriction of $\delta_A$ to only $a$-transitions by $\delta_A^a$.  For transition matrices $A,B,C$, we denote the underlying state sets of the automata by $\mathcal{A,B,C}$.  We now state the theorems of KA which we will use to reason about automata.

\section {Useful Theorems of KA}

The completeness result of \cite{bib:ko94} uses the fact that automata can be encoded as KA terms.  To simplify proofs, we add several theorems of Kleene algebra involving automata to our list of allowable rules of inference.  For each theorem we add, it will be clear that the hypotheses of the theorem are easy to check, so proofs constructed using these new rules of inference are verifiable in polynomial time.  Several of the theorems about automata are based on the following theorems of Kleene algebra:
  
$$(x + y)^* = x^*(yx^*)^*$$
$$ay = yb \rightarrow a^*y = yb^*$$
$$x(yx)^* = (xy)^*x.$$

These are known as the {\it denesting, bisimulation}, and {\it sliding} rules, respectively.  See \cite{bib:ko94} for a proof that these rules are consequences of the KA axioms.

We now provide an overview of {\it guarded string algebras}, which are models of the KAT axioms.  For a more detailed introduction, see \cite{bib:ko96}.   Guarded string algebras play the same role for KAT that regular languages do for KA; two KAT terms $t_1$ and $t_2$ are equivalent modulo the axioms of Kleene algebra with tests if and only if they denote the same set of guarded strings.

Let $P$ and $B$ be finite alphabets.  Elements of $P$ are called atomic programs, and elements of $B$ are called primitive tests (to distinguish them from atomic elements of the Boolean algebra generated by $B$).  Guarded strings are obtained from each word $w \in P^*$ by interspersing atoms of the free Boolean algebra on $B$ among the letters of $w$ (we require
that a guarded string both begins and ends with an atom).
Let $b_1,b_2,...,b_n$ be the elements of $B$.  Recall that an atom $\alpha$ of the free Boolean algebra on $B$ is a product of the form

$$\alpha = c_1 c_2 \cdots c_n$$

\noindent where $c_i \in \{b_i, \overline{b_i}\}$ for each $i$.  We require an ordering on the literals appearing in an atom so that there is a unique string denoting each atom.  Let $A_B$ denote the set of atoms.

Given a guarded string $x$, let first($x$) be the leftmost atom of $x$, and last($x$) be the rightmost atom of $x$.  We 
define a partial concatenation operation on guarded strings, denoted $\diamond$, as follows.  
Given two guarded strings, $x$ and $y$, let $x = x'\alpha$ and $y = \beta y'$, where $\alpha = $last($x$)  and $\beta = $ first($y$).

Define
$$ x \diamond y = x'\alpha y', \text{ if } \alpha = \beta, \text{ undefined otherwise}.$$

We now give interpretations of the KAT operations on sets of guarded strings. Let $C$ and $D$ be sets of guarded strings. Define
$$ 
\begin{array}{l}
C + D = C \cup D\\
C \cdot D = \{x \diamond y~|~x \in C,~y \in D\}\\
C^0 = A_B\\
C^* = \bigcup_{n \geq 0}~C^n.\\
\end{array}
$$

\noindent We must also interpret the complementation function.  Let $C$ be a set of guarded strings such that $C \subseteq A_B$.  Define
$$\overline{C} = A_B - C.$$

Using these operations, we can define a function $G$ from KAT terms to sets of guarded strings inductively.  The base cases are:
$$ 
\begin{array}{l}
G(0) = \emptyset\\
G(1) = \{\alpha~|~ \alpha \in A_B\}\\
G(b) = \{\alpha~|~ \alpha \rightarrow b \mbox{ is a propositional tautology}\}\\
G(p) = \{\alpha p \beta~|~ \alpha, \beta \in A_B\}.\\ 
\end{array}
$$

In \cite{bib:ko96}, the completeness of the guarded string model for the equational theory of KAT is shown by a reduction from the equational theory of KAT to the equational theory of KA.  This is achieved by transforming a KAT term $t$ into a KAT-equivalent term $t'$ such that $R(t')=G(t)$. Unfortunately, the term $t'$ may be exponentially longer than $t$.  We give an alternate construction. Given a term $t$, we construct an automaton $(u,A,v)$ such that $t = u^{\text{T}}A^*v$ modulo the axioms of KAT, and $(u,A,v)$ accepts precisely the set of guarded strings denoted by $t$.  The automaton $(u,A,v)$ will be polynomial in the size of $t$.

We need a few additional theorems of Kleene algebra in our construction.  The extra axioms satisfied by Boolean terms, particularly multiplicative idempotence and star-triviality, complicate the construction of the automaton.  We overcome these difficulties by selectively applying the Boolean axioms to Boolean terms.  That is, we first treat Boolean terms simply as words over an alphabet, and apply the lemmas below.  However, these lemmas produce automata which are not simple.  In the inductive construction in section 4.3 we then use the Boolean axioms to simplify the transition matrices.  Note, however, that the two lemmas below are theorems of Kleene algebra, and do not require the Boolean axioms.

\subsection{The KAT Concatenation Lemma}

The {\it KAT concatenation} lemma is based on the following alternate way of constructing an automaton accepting the concatenation of two languages.  The standard construction of such an automaton is to connect the accept states of the first automaton to the start states of the second with $\epsilon$-transitions.  However, we could also do the following:  for each state $i$ of $(u,A,v)$ with an outgoing $x$ transition to an accept state, and each state $j$ of $(s,B,t)$ with an incoming $y$ transition from a start state, add an $xy$ transition from $i$ to $j$.  Note that we allow $x$ and $y$ to be arbitrary elements of a Kleene algebra, not just letters in $\Sigma$.  
This construction yields an automaton accepting 
$u^{\rm{T}}A^*v s^{\rm{T}}B^*t$,
provided neither $(u,A,v)$ nor $(s,B,t)$ has a state which is both a start state and an accept state, which we can represent algebraically
as $u^{\rm{T}}v = 0, ~s^{\rm{T}}t = 0$.  This idea is the crux of the KAT concatenation lemma.  The lemma itself looks rather complicated, so we explain how it will be used.  In the construction in 5.2, we will have two $\epsilon$-free automata, 
$(u_1,A_1,v_1)$ and $(u_2,A_2,v_2)$.  Each of these automata will be the disjoint union of two automata:

$$(u_i,A_i,v_i) =
\left(
\left[
\begin{array}{l}
o_i \\ \hline
s_i \end{array}
\right],
\left[
\begin{array}{l|r}
C_i & 0\\ \hline
0 & B_i  \\ \end{array}
\right],
\left[
\begin{array}{l}
r_i \\ \hline
t_i \\ \end{array}
\right]
\right).
$$

\noindent It will be the case that neither of them accept the empty word, i.e.,
$$o_i^{\text{T}}r_i = 0$$
$$s_i^{\text{T}}t_i = 0$$

\noindent for $i = 1,2$.  The construction will require an automaton accepting 

$$L =(o_1^{\text{T}}C_1^*r_1s_2^{\text{T}}B_2^*t_2) +
(s_1^{\text{T}}B_1^*t_1o_2^{\text{T}}C_2^*r_2) + (s_1^{\text{T}}B_1^*t_1s_2^{\text{T}}B_2^*t_2).$$

\noindent Let $\Phi$ be a sequence of equations or equational implications.  The KAT concatenation lemma,

$$
\begin{array}{c}
\Phi \vdash o_1^{\text{T}}r_1 =0 ~~~~~~\Phi \vdash o_2^{\text{T}}r_2= 0~~~~~~\Phi \vdash s_1^{\text{T}}t_1 = 0~~~~~~~
\Phi \vdash s_2^{\text{T}}t_2 = 0 \\ \hline
\Phi \vdash
\left[
\begin{array}{l}
o_1 \\ 
s_1 \\ 
0 \\ 
0 \\ 
\end{array}
\right]^{\text{T}}
\left[
\begin{array}{l|r}
\begin{array}{lr}
C_1 & 0 \\ 
0 & B_1 \\ \end{array}
& 
\begin{array}{lr}
0 & C_1r_1 s_2^{\text{T}}B_2 \\ 
B_1t_1o_2^{\text{T}}C_2 & B_1t_1s_2^{\text{T}}B_2 \\ \end{array} 
\\ \hline
\begin{array}{lr}
0 ~~& ~~0 \\ 
0 ~~&~~ 0 \\ \end{array}
&
\begin{array}{lr}
C_2~~~~~~~~ &~~~~~~ 0 \\ 
0 ~~~~~~~~& ~~~~~~B_2 \\ \end{array}
\end{array}
\right]^*
\left[
\begin{array}{l}
0 \\ 
0 \\ 
r_2 \\ 
t_2 \\ \end{array}
\right] = L
\end{array}
$$

\noindent allows us to do this.

The proof is a straightforward calculation:
$$
\left[
\begin{array}{l}
o_1 \\ 
s_1 \\ 
0 \\ 
0 \\ 
\end{array}
\right]^{\text{T}}
\left[
\begin{array}{l|r}
\begin{array}{lr}
C_1 & 0 \\ 
0 & B_1 \\ \end{array}
& 
\begin{array}{lr}
0 & C_1r_1 s_2^{\text{T}}B_2 \\ 
B_1t_1o_2^{\text{T}}C_2 & B_1t_1s_2^{\text{T}}B_2 \\ \end{array} 
\\ \hline
\begin{array}{lr}
0 ~~& ~~0 \\ 
0 ~~&~~ 0 \\ \end{array}
&
\begin{array}{lr}
C_2~~~~~~~~ &~~~~~~ 0 \\ 
0 ~~~~~~~~& ~~~~~~B_2 \\ \end{array}
\end{array}
\right]^*
\left[
\begin{array}{l}
0 \\ 
0 \\ 
r_2 \\ 
t_2 \\ \end{array}
\right] =$$
$$ o_1^{\text{T}}C_1^*C_1r_1s_2^{\text{T}}B_2B_2^*t_2 + s_1^{\text{T}}B_1^*B_1t_1o_2^{\text{T}}C_2C_2^*r_2 + 
s_1^{\text{T}}B_1^*B_1t_1s_2^{\text{T}}B_2B_2^*t_2.$$

\noindent Using the hypotheses, it is easy to show that this sum is equal to $L$.  The proofs involved are of the following form:

$$o_1^{\text{T}}C_1^*r_1 = o_1^{\text{T}}(1 + C_1^*C_1)r_1$$
$$ = o_1^{\text{T}}r_1 + o_1^{\text{T}}C_1^*C_1r_1$$
$$ = o_1^{\text{T}}C_1^*C_1r_1.$$

\subsection{The KAT Asterate Lemma}

Let $(u,A,v)$ be a simple, $\epsilon$-free automaton and $\gamma$ be a regular expression.  Suppose $u^{\rm{T}}A^*v = \gamma$.  The standard construction of an automaton accepting
$\gamma\gamma^*$ proceeds by adding $\epsilon$-transitions from the accept states of $(u,A,v)$ back to its start states.  Suppose
$(u,A,v)$ has no paths of length 0 or 1 from a start state to an accept state, which we can model algebraically as $u^{\rm{T}}v = 0,
u^{\rm{T}}Av = 0$.  In this case, we can construct an automaton accepting $\gamma\gamma^*$ from $(u,A,v)$ with the following procedure:   for each state $i$ with an outgoing $x$ transition to an accept state, and each state $j$ with an incoming $y$ transition from a start state, add an $xy$ transition from $i$ to $j$.  This automaton, although not simple, accepts $\gamma\gamma^*$.  This idea is the basis of the 
{\it KAT asterate} lemma.

Suppose $(u,A,v)$ is the disjoint union of two automata, $(o,C,r)$ and $(s,B,t)$.  Also
suppose that $o^{\text{T}}C^*r \leq 1$, and $s^{\text{T}}t + s^{\text{T}}Bt = 0$, which implies $s^{\text{T}}B^*t = s^{\text{T}}B^*BBt$.  Under these conditions, we can apply the KAT asterate lemma:

$$
\begin{array}{c}
\Phi \vdash o^{\text{T}}C^*r \leq 1~~~~~~\Phi \vdash s^{\text{T}}B^*t = s^{\text{T}}B^*BBt\\ \hline
\Phi \vdash
\left(
\left[
\begin{array}{l}
o \\ \hline
s \end{array}
\right]^{\text{T}}
\left[
\begin{array}{l|r}
C & 0\\ \hline
0 & B\\ \end{array}
\right]^*
\left[
\begin{array}{l}
r \\ \hline
t \\ \end{array}
\right]
\right)^*
=
\left[
\begin{array}{l}
1 \\ \hline
s \end{array}
\right]^{\text{T}}
\left[
\begin{array}{l|r}
1 & 0\\ \hline
0 & B+Bts^{\text{T}}B \\ \end{array}
\right]^*
\left[
\begin{array}{l}
1 \\ \hline
t \\ \end{array}
\right] 
\end{array}.
$$

\noindent Note that $B + Bts^{\text{T}}B$ algebraically encodes the alternate asterate construction.

Since $(u,A,v)$ is the disjoint union of $(o,C,r)$ and $(s,B,t)$, it is easy to show that
$$u^{\text{T}}A^*v = o^{\text{T}}C^*r + s^{\text{T}}B^*t.$$

\noindent By KA axiom 10,

$$(u^{\text{T}}A^*v)^* = 1 + u^{\text{T}}A^*v(u^{\text{T}}A^*v)^*.$$

\noindent We can now substitute: 

$$ 1 + u^{\text{T}}A^*v(u^{\text{T}}A^*v)^* = 1 + (o^{\text{T}}C^*r + s^{\text{T}}B^*t)(o^{\text{T}}C^*r + s^{\text{T}}B^*t)^*.$$

\noindent By the denesting rule of Kleene algebra,
$$1 + (o^{\text{T}}C^*r + s^{\text{T}}B^*t)(o^{\text{T}}C^*r + s^{\text{T}}B^*t)^* =
1 + (o^{\text{T}}C^*r + s^{\text{T}}B^*t)(o^{\text{T}}C^*r)^*(s^{\text{T}}B^*t(o^{\text{T}}C^*r)^*)^*.$$

\noindent Since $o^{\text{T}}C^*r \leq 1, (o^{\text{T}}C^*r)^* = 1$.  We can simplify:

$$1 + (o^{\text{T}}C^*r + s^{\text{T}}B^*t)(o^{\text{T}}C^*r)^*(s^{\text{T}}B^*t(o^{\text{T}}C^*r)^*)^* =
1 + (o^{\text{T}}C^*r + s^{\text{T}}B^*t)(s^{\text{T}}B^*t)^*.$$

\noindent By distributivity and axiom 10 again,

$$1 + (o^{\text{T}}C^*r + s^{\text{T}}B^*t)(s^{\text{T}}B^*t)^* = 1 + s^{\text{T}}B^*t(s^{\text{T}}B^*t)^*.$$

At this point, we have shown that $u^{\text{T}}A^*v = 1 + s^{\text{T}}B^*t(s^{\text{T}}B^*t)^*$.  It remains to be shown
that under the assumption $s^{\text{T}}B^*t = s^{\text{T}}B^*BBt$,
$$~~~~~~~~~~s^{\text{T}}B^*t(s^{\text{T}}B^*t)^* = s^{\text{T}}(B + Bts^{\text{T}}B)^*t.~~~~~~~(1)$$

\noindent Reasoning algebraically,

$$s^{\text{T}}B^*t(s^{\text{T}}B^*t)^* = s^{\text{T}}B^*BBt(s^{\text{T}}B^*BBt)^*$$
$$= s^{\text{T}}B^*B(Bts^{\text{T}}B^*B)^*Bt$$
$$= s^{\text{T}}BB^*(Bts^{\text{T}}BB^*)^*Bt$$
$$= s^{\text{T}}B(B + Bts^{\text{T}}B)^*Bt.$$

\noindent The following equation is an easy consequence of the axioms of Kleene algebra:
$$(B + Bts^{\text{T}}B)^* = 1 + Bts^{\text{T}}B(B + Bts^{\text{T}}B)^* + (B + Bts^{\text{T}}B)^*Bts^{\text{T}}B + 
B(B + Bts^{\text{T}}B)^*B.$$

\noindent Multiplying the equation on the left by $s^{\text{T}}$, on the right by $t$, and simplifying using $s^{\text{T}}t=0$
and $s^{\text{T}}Bt = 0$ yields 

$$s^{\text{T}}(B + Bts^{\text{T}}B)^*t = s^{\text{T}}B(B + Bts^{\text{T}}B)^*Bt.$$

\noindent This proves (1).  We now add the trivial one-state automaton to the automaton $(s,B + Bts^{\text{T}}B,t)$, completing the proof of the KAT asterate lemma.

\section{KAT Term to Automaton}

In this section, we give the transducer which takes as input a KAT term $t$ and outputs an automaton accepting $G(t)$.  Before constructing the automaton, it must convert $t$ into a well-behaved form.

\subsection{Only Complement Primitive Tests}

The machine first uses the De Morgan laws and the Boolean axiom $\overline{\overline{b}} = b$ to transform a term $t$ into an equivalent term $t'$ in which the complementation symbol is only applied to atomic tests.  If we interpret $t'$ as a regular expression, then $R(t') \subseteq (P \cup B \cup \overline{B})^*$, where $\overline{B} = \{\overline{b}~|~ b \in B\}$.  The transducer works as follows.  On input $t$, it copies $t$ onto its worktape and onto the output tape.  Then, starting at the root of the syntax tree of $t$, it works it way down the tree until it finds a subtree containing only Boolean terms such that either some term is complemented twice, or a conjunction or disjunction appears under the complement symbol.  It then applies the appropriate axiom to this subtree, overwrites its worktape contents, and then outputs the updated term.  The machine then begins searching again at the root of the tree.  When it scans the whole tree and does not have to apply any axioms, it stops.  The transducer requires only polynomially many worktape cells.  Furthermore, the increase in the size of the term is negligible. At the end of this stage, it has $t'$ written on its worktape.

\subsection{New variables for atoms}

For the remainder of the construction, it is advantageous to treat each atom as a single letter.  Let $z = 2^{|B|}$.  The machine generates
$z$ many new variables, $x_1,x_2,...,x_z$.  For each $i$, it outputs the equation 
$$x_i = \alpha_i$$
\noindent where $\alpha_i$ is the $i^{\text{th}}$ atom.  The automaton constructed below uses the alphabet 
$P \cup \{x_1,x_2,...,x_z\}$.  It is a routine matter to verify that two KAT terms denote same set of guarded strings if and only if they denote the same set of words after performing this substitution.  For the rest of the construction, we use the terms ``guarded strings'' and  ``guarded strings after this substitution'' interchangeably.  

\subsection{Constructing the Automaton}

Now that the preprocessing of the term is complete, the machine constructs the automaton.  The construction is inductive and resembles the construction the proof of Kleene's theorem.  However, the machine will maintain several invariants throughout the construction which were not necessary in the pure Kleene algebra case. At a given substage, let $(u,A,v)$ be the final automaton constructed.  The automaton $(u,A,v)$ will satisfy:

\begin{itemize}
\item $(u,A,v)$ is simple and $\epsilon$-free.
\item $(u,A,v)$ is the ``disjoint union'' of two automata, $(o,C,r)$ and $(s,B,t)$, or just $(o,C,r)$, or just $(s,B,t)$.
\item $(s,B,t)$ accepts only words of length two or more, so., $s^{\text{T}}B^*t = s^{\text{T}}B^*BBt$.
\item $(o,C,r)$ is a two state automaton accepting only one-letter words from the alphabet $\{x_1,x_2,...,x_z\}$.
\item The first two states of $(u,A,v)$ are the states of $(o,C,r)$ (if $(o,C,r)$ is nonempty).
\end{itemize}

The base case of the induction is as follows.  For an atomic term $a$, $\hat{a}$ denotes the automaton constructed.  For an atom $x_i$
and a primitive test $b$, $x_i \leq b$ means that $x_i \rightarrow b$ is a propositional tautology.
$$
\begin{array}{l}
\hat{0} = (0,0,0)\\
\\
\hat{1} = 
\left(
\left[
\begin{array}{c}
1 \\
0 \\ \end{array}
\right],
\left[
\begin{array}{lr}
0 & \sum_i x_i \\
0 & 0  \\ \end{array}
\right],
\left[
\begin{array}{c}
0 \\
1 \\ \end{array}
\right]
\right)\\
\\
\hat{b} = 
\left(
\left[
\begin{array}{c}
1 \\
0 \\ \end{array}
\right],
\left[
\begin{array}{lr}
0 & \sum_{x_i \leq b} x_i\\
0 & 0  \\ \end{array}
\right],
\left[
\begin{array}{c}
0 \\
1 \\ \end{array}
\right]
\right)\\
\\
\hat{p} =
\left(
\left[
\begin{array}{c}
1 \\
0 \\ 
0 \\
0 \\ \end{array}
\right],
\left[
\begin{array}{cccc}
0 & \sum_i x_i & 0 & 0 \\
0 & 0 & p & 0\\ 
0 & 0 & 0 & \sum_i x_i\\
0 & 0 & 0 & 0\end{array}

\right],
\left[
\begin{array}{c}
0 \\
0 \\
0 \\
1 \\ \end{array}
\right]
\right)\\
\end{array}
$$

\noindent For each automaton, the machine must prove that the language it accepts is KAT-equivalent to the appropriate atomic term.  There are
finitely many atomic terms, so the machine can store all of the necessary proofs in its finite control.  Note that this expansion increases the size of a term by only a constant amount, although the constant is exponential in $|B|$.  Cf. the proof that the Boolean algebra axioms entail all propositional tautologies.

We now treat the inductive step of the construction.  The easiest automaton to construct is that for addition.  Suppose we have two automata $(u_1,A_1,v_1)$ and $(u_2,A_2,v_2)$, such that $u_1^{\text{T}}A_1^*v_1 = \gamma$ and $u_2^{\text{T}}A_2^*v_2 = \delta$.
By induction, $(u_1,A_1,v_1)$ is the disjoint union of $(o_1,C_1,r_1)$ and $(s_1,B_1,t_1)$, and 
$(u_2,A_2,v_2)$ is the disjoint union of $(o_2,C_2,r_2)$ and $(s_2,B_2,t_2)$.  The machine first proves the equations

$$u_1^{\text{T}}A_1^*v_1 = o_1^{\text{T}}C_1^*r_1 + s_1^{\text{T}}B_1^*t_1$$
$$u_2^{\text{T}}A_2^*v_2 = o_2^{\text{T}}C_2^*r_2 + s_2^{\text{T}}B_2^*t_2.$$

\noindent It then outputs a proof that
$$\gamma + \delta = (o_1^{\text{T}}C_1^*r_1 + o_2^{\text{T}}C_2^*r_2) +  s_1^{\text{T}}B_1^*t_1 + s_2^{\text{T}}B_2^*t_2.$$

\noindent The machine can now construct a two-state automaton $(o,C,r)$ which accepts $(o_1^{\text{T}}C_1^*r_1 + o_2^{\text{T}}C_2^*r_2)$, then apply the addition construction from 4.1 to $(o,C,r),(s_1,B_1,t_1)$, and $(s_2,B_2,t_2)$.  This yields an automaton $(u,A,v)$ which satisfies the invariants and accepts $\gamma + \delta$.  Note that there are only finitely many possibilities for $(o_1,C_1,r_1)$ and 
$(o_2,C_2,r_2)$, so the machine can prove

$$o^{\text{T}}C^*r = o_1^{\text{T}}C_1^*r_1 + o_2^{\text{T}}C_2^*r_2$$

\noindent using data from its finite control.
 
The automaton for the product of two terms is more complicated.  Again, let $(u_1,A_1,v_1)$ and $(u_2,A_2,v_2)$ be two automata such that
$u_1^{\text{T}}A_1^*v_1 = \gamma$ and $u_2^{\text{T}}A_2^*v_2 = \delta$.
As in the case for addition, we use the fact that each of these automata is the disjoint union of two automata:

$$u_1^{\text{T}}A_1^*v_1 = o_1^{\text{T}}C_1^*r_1 + s_1^{\text{T}}B_1^*t_1$$
$$u_2^{\text{T}}A_2^*v_2 = o_2^{\text{T}}C_2^*r_2 + s_2^{\text{T}}B_2^*t_2.$$

\noindent The machine can output a proof of the equations

$$\gamma \delta = (o_1^{\text{T}}C_1^*r_1 + s_1^{\text{T}}B_1^*t_1)(o_2^{\text{T}}C_2^*r_2 + s_2^{\text{T}}B_2^*t_2)$$
$$ = (o_1^{\text{T}}C_1^*r_1o_2^{\text{T}}C_2^*r_2) +(o_1^{\text{T}}C_1^*r_1s_2^{\text{T}}B_2^*t_2) +
(s_1^{\text{T}}B_1^*t_1o_2^{\text{T}}C_2^*r_2) + (s_1^{\text{T}}B_1^*t_1s_2^{\text{T}}B_2^*t_2).$$

The term $(o_1^{\text{T}}C_1^*r_1o_2^{\text{T}}C_2^*r_2)$ is a sum of atoms after simplifying using the Boolean axioms.  The machine
can construct a two-state automaton $(o,C,r)$ accepting this sum.  Since there are only finitely many choices for
$o_1^{\text{T}}C_1^*r_1$ and $o_2^{\text{T}}C_2^*r_2$, all of the necessary proofs can be stored in the finite control of the machine.

Let $(s,B,t)$ be the automaton

$$
\left(
\left[
\begin{array}{l}
o_1 \\ 
s_1 \\ 
0 \\ 
0 \\ 
\end{array}
\right],
\left[
\begin{array}{l|r}
\begin{array}{lr}
C_1 & 0 \\ 
0 & B_1 \\ \end{array}
& 
\begin{array}{lr}
0 & C_1r_1 s_2^{\text{T}}B_2 \\ 
B_1t_1o_2^{\text{T}}C_2 & B_1t_1s_2^{\text{T}}B_2 \\ \end{array} 
\\ \hline
\begin{array}{lr}
0 ~~& ~~0 \\ 
0 ~~&~~ 0 \\ \end{array}
&
\begin{array}{lr}
C_2~~~~~~~~ &~~~~~~ 0 \\ 
0 ~~~~~~~~& ~~~~~~B_2 \\ \end{array}
\end{array}
\right],
\left[
\begin{array}{l}
0 \\ 
0 \\ 
r_2 \\ 
t_2 \\ \end{array}
\right]
\right).
$$

\noindent The machine first outputs proofs of the hypotheses of the KAT concatenation lemma.
It can then output

$$s^{\text{T}}B^*t = (o_1^{\text{T}}C_1^*r_1s_2^{\text{T}}B_2^*t_2) +
(s_1^{\text{T}}B_1^*t_1o_2^{\text{T}}C_2^*r_2) + (s_1^{\text{T}}B_1^*t_1s_2^{\text{T}}B_2^*t_2),$$

\noindent which follows from the KAT concatenation lemma.

The machine now constructs a simple automaton $(s,B',t)$ by simplifying the transition matrix for $(s,B,t)$ using the Boolean axioms and
outputs a proof of the equivalence of $(s,B,t)$ and $(s,B',t)$.  It then adds the automata $(o,C,r)$ and $(s,B',t)$ together to get $(u,A,v)$, and outputs a proof of the equation
$$u^{\text{T}}A^*v = \gamma\delta.$$

Finally, we come to the construction for $^*$.  Let $(u,A,v)$ be an automaton such that $u^{\text{T}}A^*v = \gamma$.  This automaton
is the disjoint union of two automata, $(o,C,r)$ and $(s,B,t)$ such that $(o,C,r)$ accepts a sum of atoms and $(s,B,t)$ accepts no words of length less than two.  The machine first outputs proofs that
$$o^{\text{T}}C^*r \leq 1$$
$$s^{\text{T}}B^*BBt = s^{\text{T}}Bt.$$

\noindent These facts follow from the Boolean axioms and the equation $s^{\text{T}}t + s^{\text{T}}Bt = 0$.

The machine can now output
$$
\left[
\begin{array}{l}
1 \\ \hline
s \end{array}
\right]^{\text{T}}
\left[
\begin{array}{l|r}
1 & 0\\ \hline
0 & B+Bts^{\text{T}}B \\ \end{array}
\right]^*
\left[
\begin{array}{l}
1 \\ \hline
t \\ \end{array}
\right]
= \gamma^*,
$$

\noindent which follows from the KAT asterate lemma.  Finally, the machine can apply the Boolean axioms to each entry of
$$
\left[
\begin{array}{l|r}
1 & 0\\ \hline
0 & B+Bts^{\text{T}}B \\ \end{array}
\right]
$$

\noindent to produce an equivalent simple, $\epsilon$-free transition matrix $D$ (1 becomes the sum of all atoms).  It can then output a proof of
$$
\left[
\begin{array}{l}
1 \\ \hline
s \end{array}
\right]^{\text{T}}
\text{{\Large D}}^*
\left[
\begin{array}{l}
1 \\ \hline
t \\ \end{array}
\right]
= \gamma^*.
$$

The proof that the automaton constructed for a term $t$ accepts precisely the guarded strings denoted by $t$ is a straightforward induction.

\section{Reducing the Hoare Theory of KA(T) to the Equational Theory of KA}

Finally, we make the simple observation that the reductions in \cite{bib:co93} and \cite{bib:ko96} don't significantly increase the size of the terms.

\newtheorem{theorem}{Theorem}
\begin{theorem}
{\it Proofs of equational implications in the Hoare Theory of KA(T) can be produced by a PSPACE transducer.}
\end{theorem}

\begin{proof}  Given an alphabet $\Sigma = \{a_1,a_2,...,a_n\}$, let $u = a_1 + a_2 + \cdots + a_n$.  In \cite{bib:co93}, it is shown that

$$s \equiv t \Leftrightarrow s + uru = t + uru$$

\noindent is a Kleene algebra congruence, therefore $(r = 0 \rightarrow p = q) \leftrightarrow (p + uru = q + uru)$.  The same reduction works for KAT, as is shown in \cite{bib:ko96} - in this case $u$ is only defined to be the sum of all of the atomic programs, not the atomic tests.  The transformation from $r = 0 \rightarrow p = q$ to $p + uru = q + uru$ involves only a constant increase in size.
\end{proof}

\section{Deterministic {\tt while} Programs}

Let $P$ be a set of atomic programs, and $B$ be a set of atomic tests.  In \cite{bib:ko97}, it is shown how to encode deterministic {\tt while} programs as KAT terms:
$$p;q = pq$$
$${\bf if}~b~ {\bf then}~p ~{\bf else}~q = bp + \overline{b}q$$
$${\bf if}~b~ {\bf then}~p = bp + \overline{b}$$
$${\bf while}~b~{\bf do}~p = (bp)^*\overline{b}.$$
Let $t$ be a KAT term which is an encoding of a deterministic ${\tt while}$ program.  Let $g$ be a guarded string over $(P \cup A_B)$.  It is easy to see that the automaton $(u,A,v)$ constructed from $t$ in section 4 satisfies the following:
\begin{itemize}
\item There is only one start state $s$ of $(u,A,v)$ with an outgoing transition labeled by an atom $x$ such that first$(g) =x$.
\item $|\hat{\delta}_A(s,g)| \leq 1$.
\end{itemize}

Therefore, when considering the deterministic automaton $(s,D,t)$ obtained from $(u,A,v)$ by the standard subset construction, all states of $(s,D,t)$ corresponding to more than one state of $(u,A,v)$ are inaccessible.  This implies that, given two KAT terms $t_1$ and $t_2$, using the above procedure to construct automata for each term and then using the procedure in \cite{bib:wo08} to generate proofs of equivalence of the automata yields proofs which are only polynomial in $|t_1| + |t_2|$.

\end{document}